\documentclass[11pt,english]{article}

\usepackage[margin = 2.5cm]{geometry}

\usepackage{amsthm}
\usepackage{amsmath}
\usepackage{amssymb}
\usepackage{setspace}
\usepackage{mathtools}
\usepackage{graphicx}
\usepackage[hidelinks]{hyperref}
\usepackage{thm-restate}
\usepackage{cleveref}

\usepackage{enumerate}
\usepackage{framed}
\usepackage{subcaption}

\theoremstyle{plain}

\newtheorem*{thm*}{Theorem}
\newtheorem{thm}{Theorem}
\Crefname{thm}{Theorem}{Theorems}

\newtheorem*{lem*}{Lemma}

\Crefname{lem}{Lemma}{Lemmas}

\newtheorem*{claim*}{Claim}

\crefname{claim}{Claim}{Claims}
\Crefname{claim}{Claim}{Claims}

\newtheorem{prop}[thm]{Proposition}
\Crefname{prop}{Proposition}{Propositions}

\newtheorem{cor}[thm]{Corollary}
\Crefname{cor}{Corollary}{Corollaries}
\crefname{cor}{Corollary}{Corollaries}

\newtheorem{conj}[thm]{Conjecture}
\Crefname{conj}{Conjecture}{Conjectures}

\newtheorem{qn}[thm]{Question}
\Crefname{qn}{Question}{Questions}

\newtheorem{obs}[thm]{Observation}
\Crefname{obs}{Observation}{Observations}

\theoremstyle{definition}

\Crefname{prob}{Problem}{Problems}

\Crefname{defn}{Definition}{Definitions}

\theoremstyle{remark}

\renewenvironment{proof}[1][]{\begin{trivlist}
\item[\hspace{\labelsep}{\bf\noindent Proof#1.\/}] }{\qed\end{trivlist}}

\DeclareMathOperator{\modulu}{mod}

\newcommand{\F}{\mathcal{F}}
\newcommand{\G}{\mathcal{G}}
\newcommand{\MH}{\mathcal{H}}
\renewcommand{\P}{\mathcal{P}}

\renewcommand{\mod}{\;\modulu\,}

\newcommand{\remark}[1]{} % comment to get the remarks back

\newcommand{\remove}[1]{}

\begin{document}

\title{Almost partitioning the hypercube into copies of a graph}
\author{
    Vytautas Gruslys\thanks{
        Department of Pure Mathematics and Mathematical Statistics, 	
        University of Cambridge, 
        Wilberforce Road, 
        CB3\;0WB Cambridge, 
        UK;
        e-mail:
        \texttt{v.gruslys@dpmms.cam.ac.uk}.
    }
    \and
    Shoham Letzter\thanks{	
		ETH Institute for Theoretical Studies,
		ETH Zurich,
		8092 Zurich,
		Switzerland;
		e-mail:
		\texttt{shoham.letzter@math.ethz.ch}. 
	}
}

\maketitle

\begin{abstract}

    \setlength{\parskip}{\medskipamount}
    \setlength{\parindent}{0pt}
    \noindent
	Let $H$ be an induced subgraph of the hypercube $Q_k$, for some $k$.  We
	show that for some $c = c(H)$, the vertices of $Q_n$ can be partitioned
	into induced copies of $H$ and a remainder of at most $O(n^c)$
	vertices.  We also show that the error term cannot be replaced by
	anything smaller than $\log n$.

\end{abstract}

\section{introduction} \label{sec:intro} 
	Given graphs $G$ and $H$, an \emph{$H$-packing of $G$} is a collection of
	vertex-disjoint copies of $H$ in $G$. A \emph{perfect $H$-packing} (also
	known as an $H$-factor) is an $H$-packing that covers all the vertices of
	the ground graph $G$ (so, in order for $G$ to have a perfect $H$-packing,
	$|H|$ must divide $|G|$).  A natural question asks for conditions on $G$
	that imply the existence of an $H$-factor. For example, a well researched
	question asks for the smallest minimum degree that implies the existence
	of an $H$-factor.  If $H$ is an edge (and more generally if $H$ is a
	path), then, by Dirac's theorem \cite{Dirac}, if $G$ has $n$ vertices and
	minimum degree at least $n/2$ (and $|H|$ divides $|G|$), then $G$ has a
	perfect $H$-packing.  Corr\'adi and Hajnal \cite{Corradi-Hajnal} showed
	that $\delta(G) \ge 2n/3$ guarantees the existence of a perfect
	$K_3$-packing and Hajnal and
	Szemer\'edi \cite{Hajnal-Szemeredi} extended this result by
	showing that if $\delta(G) \ge (1 - 1/r)n$ then $G$ has a perfect
	$K_r$-packing. We remark that these conditions on $\delta(G)$ are best
	possible.

	After a series of papers by Alon and Yuster
	\cite{Alon-YusterI,Alon-YusterII} and by Koml\'os, S\'ark\"ozy
	and Szemer\'edi \cite{KSS}, Kuhn and Osthus \cite{Kuhn-Osthus} found the
	smallest minimum degree condition that guarantees the existence of an
	$H$-factor, up to an additive constant error term, and for all $H$.

	We consider a different problem, where instead of looking for
	$H$-packings in graphs of large minimum degree, we focus on $H$-packings
	of the hypercube $Q_n$. There are two obvious conditions for the
	existence of a perfect $H$-packing in $Q_n$: $H$ has to be a subgraph of
	$Q_n$; and the order of $H$ has to be a power of $2$. Gruslys
	\cite{Gruslys16III} showed that these two conditions are sufficient for
	large $n$, thus confirming a conjecture of Offner \cite{Offner}. In fact,
	he showed that if $H$ is an induced subgraph of $Q_k$ for some $k$ and
	$|H|$ is a power of $2$, then there is a perfect packing of $G$ into
	\emph{induced} copies of $H$.

	A similar problem concerns packings of the Boolean lattice $2^{[n]}$ into
	induced copies of a poset $P$. Note that here, if we drop the induced
	condition, we reduce to the case where $P$ is a chain, thus in the case
	of posets we only consider induced copies of $P$.
	Again, there are two obvious necessary conditions: $P$ must have a
	minimum and maximum elements; and the order of $P$ has to be a power of
	$2$. Lonc \cite{Lonc} conjectured that for large enough $n$, these
	conditions are also sufficient, and verified the conjecture for the case
	where $P$ is a chain. This conjecture was recently solved by Gruslys,
	Leader and Tomon \cite{Gruslys16II}. 
	
	\remove{
	We note that both this result and
	the aforementioned result of Gruslys \cite{Gruslys16III} use a general
	method for finding perfect packings of product spaces. This method was
	introduced by Gruslys, Leader and Tan \cite{Gruslys16I} (where the
	authors solved a similar tiling problems for $\mathbb{Z}^n$) and Gruslys,
	Leader and Tomon \cite{Gruslys16II} extended the method to be applicable
	for product spaces.
	}

	It is natural to ask what can be said when the divisibility
	condition does not hold. 
	Gruslys, Leader and Tomon \cite{Gruslys16II}  
	conjectured that if $P$ is a poset with a maximum and a minimum, then
	there is  a $P$-packing of
	$Q_n$ that covers all but at most $c$ elements, where $c = c(P)$ is a
	constant that depends on $P$. This conjecture was
	recently proved by Tomon \cite{Tomon16II}.

	\remark{should I mention the special case of chains? Griggs conjectured
	and Lonc solved it}

	In light of this result, it is natural to ask if a
	similar phenomenon holds in the case of $H$-packings of the hypercube
	$Q_n$. Namely, if $H$ is a subgraph of $Q_k$ for some $k$, how large an
	$H$-packing of $Q_n$ can we find? As our first main result, we show that
	if $H$ is a subgraph of $Q_k$ then there is an $H$-packing of $Q_n$ that
	covers all but at most $O(n^c)$ vertices, where $c = c(H)$.

	\begin{restatable}{thm}{thmAlmostTiling} \label{thm:almost-tiling}
		Let $H$ be an induced subgraph of $Q_k$ for some $k$. Then there
		exists a packing of $G$ by induced copies of $H$, such that at most
		$O(n^c)$ vertices remain uncovered, where $c = c(H)$.	
	\end{restatable}

	It is natural to wonder if the number of uncovered vertices can be
	reduced to be at most $c(H)$. Perhaps surprisingly, it turns out that this is not
	always the case. As our second main result, we show that a
	$(P_3)^3$-packing of $Q_n$ misses at least $\log n$ vertices ($P_3$ is
	the path on three vertices).

	\begin{restatable}{thm}{thmMissingVs} \label{thm:missing-vs}
		In every $(P_3)^3$-packing of $Q_n$, at least $\log n$ points
		are uncovered.
	\end{restatable}

	\subsection{Notation}

		We denote the path on $l$ vertices by $P_l$. 
		When we say that $G$ can be partitioned into copies of $H$, we mean
		that there exists a perfect $H$-packing of $G$.
		For graphs $G_1$ and $G_2$, we denote by $G_1 \times
		G_2$ the Cartesian product of $G_1$ and $G_2$, which has vertex set
		$V(G_1) \times V(G_2)$ and $((u_1, v_1), (u_2, v_2))$
		is an edge iff $u_1 = v_1$ and $u_2 v_2 \in E(G_2)$ or $u_2 = v_2$ and
		$u_1 v_1 \in E(G_1)$. The $n$-th power $G^n$ of $G$ is defined to be $G
		\times \ldots \times G$ (where $G$ appears $n$ times).

	\subsection{Structure of the paper}

		This paper consists of three parts.  In the first part (see
		\Cref{sec:one-mod-l}) we prove that if $H$ is an induced subgraph of
		$Q_k$ for some $k$, then for sufficiently large $n$, there is a
		perfect packing of $(P_{2|H|})^n$ into induced copies of $H$ (see
		\Cref{thm:tiling-powers-paths}).  In the second part
		(\Cref{sec:tiling-with-powers-paths}), we prove that there is a
		packing of $Q_n$ by induced copies of $(P_l)^t$ which leaves at most
		$O(n^{t-1})$ vertices uncovered.  These two parts easily combine to
		form a
		proof of \Cref{thm:almost-tiling}.  Finally, in the third part
		(\Cref{sec:missing-vs}) we prove \Cref{thm:missing-vs}, thus showing
		that the error term $O(n^{t-1})$ cannot be replaced by something
		smaller than $\log n$.

		Before proceeding to the proofs, we give an overview of them in
		\Cref{sec:overview}. We finish the paper with concluding remarks 
		and open problems in \Cref{sec:conclusion}.

\section{Overview of the proofs} \label{sec:overview}

	In this section we give an overview of the proofs in this paper.
	
	\subsection{Partitioning $(P_{2l})^n$}
		Our first aim in this paper is to prove
		\Cref{thm:tiling-powers-paths}.

		\begin{restatable}{thm}{thmTilingPowersPaths}
			\label{thm:tiling-powers-paths}
			Let $H$ be an induced subgraph of $Q_k$ for some $k$. Then
			$(P_{2|H|})^n$ can be partitioned into induced copies of $H$,
			whenever $n$ is sufficiently large.
		\end{restatable}

		Our proof follows the footsteps of Gruslys \cite{Gruslys16III} who
		proved that if $H$ is an induced subgraph of a hypercube whose order is a
		power of $2$, then for large $n$ there is a perfect packing of $Q_n$
		into induced copies of $H$.
		 
		An important tool in the proof of \Cref{thm:tiling-powers-paths} is a
		result of Gruslys, Leader and Tomon \cite{Gruslys16II} (introduced by
		Gruslys, Leader and Tan \cite{Gruslys16I} for tiling of
		$\mathbb{Z}^n$) which gives a
		a general method for proving the existence of perfect packings of a
		product space $A^n$ into copies of a subset $S$ of $A$. Given a
		subset $S$ of $A$, a collection of copies of $S$ in $A^n$ (which may
		contain a certain copy several times) is called an $l$-partition ($(r
		\mod l)$-partition) if every vertex in $A^n$ is covered by exactly
		$l$ ($(r \mod l)$) copies of $S$.  We note that a $1$-partition is
		simply a perfect packing. Trivially, the existence of a perfect
		packing implies the existence of an $l$-partition and a $(1 \mod
		l)$-partition. Remarkably, the aforementioned result of Gruslys,
		Leader and Tomon shows that, roughly speaking, the opposite is true.
		Namely, they showed that if there exists $l$ for which $A^m$ admits
		an $l$-partition and a $(1 \mod l)$-partition into copies of $S$,
		then, for large $n$, $A_n$ admits a perfect packing into copies of
		$S$. The precise statement of this result is given in
		\Cref{thm:general-tiling}.

		The existence of an $|H|$-partition of $(P_{2|H|})^n$ into induced
		copies of $H$ is a simple observation (see \Cref{obs:l-tiling}).
		The existence of a $(1 \mod |H|)$-partition of $(P_{2|H|})^n$
		into copies of $H$ is more difficult to prove, but it is quite
		straightforward to adapt the methods of Gruslys \cite{Gruslys16III}
		to work in our setting. These two facts, together the aforementioned
		result \cite{Gruslys16II}, form the proof of
		\Cref{thm:tiling-powers-paths}.

	\subsection{Almost partitioning $Q_n$ into powers of a path}

		Our second aim is to prove \Cref{thm:almost-tiling-induced-paths}.
		\begin{restatable}{thm}{thmAlmostTilingInducedPaths}
			\label{thm:almost-tiling-induced-paths}
			For any $l$ and $t$, there is a packing of $Q_n$ into induced
			copies of $(P_l)^t$, for which at most
			$O(n^{t-1})$ vertices are uncovered. 
		\end{restatable}

		The fact that $Q_n$ is Hamiltonian shows that there is a
		$P_l$-packing of $Q_n$ missing fewer than $l$ vertices.
		In fact, if $l$ divides $2^n - 1$, then exactly one vertex remains
		uncovered. This observation allows us to prove the existence of a
		$(P_l)^t$-packing of $Q_n$ with at most $O(n^{t-1})$ uncovered
		vertices, whenever $l$ is odd (see \Cref{obs:tiling-odd-paths}).
		It is then not hard to conclude that the same holds for all $l$ (see
		\Cref{cor:tiling-even-paths}) using the observation that $(P_{2l})^t$ is
		a subgraph of $(P_l)^t \times Q_t$.

		Note that this does not imply \Cref{thm:almost-tiling-induced-paths},
		since we require that the copies of $(P_l)^t$ are induced.
		We notice that if
		$H$ is a graph on $l$ vertices with a Hamilton path, then $H \times
		P_{l-1}$ has a perfect packing into induced $P_l$'s (see
		\Cref{obs:hamilton-times-induced-path}). This fact, with a little
		more work, allows us to use the packing of $Q_n$ into (not necessarily
		induced) copies of $(P_{l'})^t$ to obtain a packing into induced copies
		of $(P_l)^t$ (where $l'$ is suitably chosen).

		We note that \Cref{thm:almost-tiling} follows from
		\Cref{thm:tiling-powers-paths,thm:almost-tiling-induced-paths}.

		\begin{proof}[ of \Cref{thm:almost-tiling}]
			Let $H$ be a subgraph of $Q_k$. Then by
			\Cref{thm:tiling-powers-paths}, there exists $m$ for which there
			is a perfect packing of 
			$(P_{2|H|})^m$ into induced copies of $H$.
			By \Cref{thm:almost-tiling-induced-paths}, 
			there is a packing of $Q_n$ into induced copies of $(P_{2|H|})^m$,
			such that at most 
			$O(n^{m-1})$ vertices  are uncovered.
			Hence there exists a packing of $Q_n$ into induced copies of $H$
			with at most $O(n^{m-1})$ uncovered vertices (note that $m$
			depends only on $H$).
		\end{proof}

	\subsection{Lower bound on the number of uncovered vertices}
		
		Our final aim is to prove \Cref{thm:missing-vs}.
		\thmMissingVs*

		We use the properties of $Q_n$ and of $(P_3)^3$ to conclude that the
		size of the intersection of any co-dimension-$2$ subcube of $Q_n$
		with any copy of $(P_3)^3$ is divisible by $3$. In fact, we deduce
		this from a similar statement for $(P_3)^t$ (see
		\Cref{prop:intersect-codim-one}).  We conclude that the set of
		uncovered vertices in a $(P_3)^3$-packing of $Q_n$ forms a
		\emph{separating family} for $[n]$, implying that it has size at
		least $\log n$.

\section{Perfect $H$-packings of $(P_{2|H|})^n$} \label{sec:one-mod-l}

	Our main aim in this section is to prove
	\Cref{thm:tiling-powers-paths}.
	
	\thmTilingPowersPaths*

	Recall that a result of Gruslys, Leader and Tomon \cite{Gruslys16II}
	implies that it suffices to find $l$- and $(1 \mod l)$-partitions into
	copies of $H$. Before stating their result precisely, we introduce
	some notation. 

	Let $A$ be a set. We identify $A^n
	\times A^m$ with $A^{n + m}$ (whenever $m$ and $n$ are positive
	integers). Thus, for any $x \in A^n$ and $y \in A^m$, we treat $(x, y)$
	as an element of $A^{n + m}$.
	Given a set $X$ in $A^n$, and a permutation $\pi : [n] \rightarrow
	[n]$, we define $\pi(X)$ to be the image of $X$ under the permutation
	of the coordinates according to $\pi$.
	In other words, $\pi(X) = \left\{ \left(x_{\pi(1)}, \ldots, x_{\pi(n)}
	\right) : (x_1, \ldots, x_n) \in X \right\}$.
	Finally, given sets $X$ in $A^m$ and $Y$ in $A^n$ where $m \le n$, we
	say that $Y$ is
	a \emph{copy} of $X$ if $Y = \pi(X \times \{y\})$ for some $y \in A^{n -
	m}$.
	\remark{we are using a different definition of a copy in the paper}

	\begin{thm} [Gruslys, Leader, Tomon \cite{Gruslys16II}]
		\label{thm:general-tiling}
		Let $\F$ be a family of subsets of a finite set $A$. If there
		exists $l$ for which $\F$ contains an $l$-partition and a $(1
		\mod l)$-partition of $S$, then there exists $n$ for which $S^n$
		admits a partition into copies of elements in $\F$.
	\end{thm}
	
	The task of finding an $l$-partition is quite simple. In fact, it
	follows directly from the analogous result in \cite{Gruslys16III}
	and the fact that $(P_{2l})^n$ can be partitioned into copies of
	$Q_n$. For the sake of completeness, we include the proof here.

	\begin{obs} \label{obs:l-tiling}
		Let $H$ be an induced subgraph of $Q_k$ for some $k$.
		Then there is an $|H|$-partition of $(P_{2|H|})^n$ into induced
		copies of $H$, for any $n \ge k$.
	\end{obs}

	\begin{proof}
		Denote $l = |H|$.
		Note that, since the path $P_{2l}$ can be partitioned into $l$
		edges, its $n$-th power $(P_{2l})^n$ can be partitioned into
		$l^n$ induced copies of $Q_n$. Thus, it suffices to exhibit an
		$l$-partition of $Q_n$ into induced copies of $H$.
		Let $X$ be the vertex set of some induced copy of $H$ in $Q_n$.
		We consider the set of all \emph{shifts} of $X$. For every $u \in
		Q_n$, we note that the set $X + u = \{x + u : x \in X\}$
		(addition is done coordinate-wise and modulo $2$) is an induced
		copy of $H$ in $Q_n$.
		Consider the collection $\{X + u : u \in Q_n\}$. By symmetry,
		every vertex in $Q_n$ is covered by the same number of sets.
		Furthermore, there are $2^n$ such sets, each covers $l$ points,
		so the number of times each vertex is covered is $\frac{2^n
		l}{2^n} = l$. So, we found an $l$-partition of $Q_n$ into
		induced copies of $H$.
	\end{proof}

	The next task, of finding a $(1 \mod l)$-partition of $(P_{2l})^n$
	into copies of $H$, is significantly harder.		
	Unlike \Cref{obs:l-tiling}, we cannot directly apply the analogous
	result of Gruslys \cite{Gruslys16III}. Instead, we adapt his method
	to our setting. 

	\begin{restatable}{thm}{thmOneModLTiling} \label{thm:one-mod-l-tiling}
		Let $H$ be a non-empty induced subgraph of $Q_k$ for some $k$.
		Then there is a $(1 \mod l)$-partition of $(P_{2l})^k$ into induced
		copies of $H$.
	\end{restatable}

	We note that in \Cref{thm:tiling-powers-paths},
	unlike \Cref{obs:l-tiling}, there
	is no restriction on the order of $H$.
	
	\remark{spacing before mod}
	
	Before proceeding to the proof of \Cref{thm:one-mod-l-tiling}, we
	show how to prove \Cref{thm:tiling-powers-paths} using
	\Cref{obs:l-tiling,thm:one-mod-l-tiling}.

	\begin{proof} [ of \Cref{thm:tiling-powers-paths}]
		Denote $l = |H|$.
		Note that it suffices to show that for some $n$ the graph 
		$(P_{2l})^n$ can be partitioned into induced copies of $H$.
		Recall that $H$ is an induced subgraph of $Q_k$, for some $k$. By
		\Cref{thm:one-mod-l-tiling}, there is a $(1 \mod l)$-partition of
		$A = (P_{2l})^k$ into induced copies of $H$. By
		\Cref{obs:l-tiling} there is an $l$-partition of $A$ into induced
		copies of $H$. Hence, by \Cref{thm:general-tiling} there exists
		$n$ for which there is a perfect $H$-packing of $A^n =
		(P_{2l})^{kn}$, as required. 
	\end{proof}
	
	We now proceed to the proof of \Cref{thm:one-mod-l-tiling}.

	\begin{proof} [ of \Cref{thm:one-mod-l-tiling}]
		We shall prove the following slightly stronger claim: if $H$ is a
		non-empty induced subgraph of $Q_k$, then there is a $(1 \mod r)$
		partition of $(P_{2l})^k$ into isometric copies of $H$.

		Let us first explain briefly what we mean by an isometric copy of $H$
		in a graph $G$.  We consider the graphs $Q_k$ and $G$ together with
		the metric coming from the graph distance. An isometric copy of $H$
		is the image of $H$ under an isometry $f: Q_k \rightarrow G$ (here we
		fix a particular embedding of $H$ in $Q_k$). 
		
		Define $H_{-}$ and $H_{+}$ as follows.
		\begin{align*}
			&H_{-} = \{u \in Q_{k-1} : (u, 0) \in H\} \\
			&H_{+} = \{u \in Q_{k-1} : (u, 1) \in H\}.
		\end{align*}

		We prove the claim by induction on $k$. It is trivial for $k = 1$
		(then $H$ is either a single vertex or an edge), so
		suppose that $k \ge 2$ and the claim holds for $k - 1$.
		Note that we may assume that $H_{-}$ and $H_{+}$ are both non-empty.
		We shall show that $(P_{2l})^k$ has a $(1 \mod l)$-partition into
		isometric copies of $H$.

		We denote the vertices of $P_{2l}$ by $\{0, 1, \ldots,
		2l-1\}$.
		Let $p \in [2l-2]$.  By induction, there is a collection of isometric
		copies of $H_{-}$ in $(P_{2l})^{k - 1} \times \{p\}$ such that each
		point is covered $(1 \mod l)$ times.  Let $A$ be the vertex set of
		such a copy of $H_{-}$. Then there is an isometric copy of $H$ in
		$(P_{2l})^{k - 1} \times \{p, p + 1\}$ whose intersection with
		$(P_{2l})^{k - 1} \times \{p\}$ is $A$.  It follows that there exists
		a collection $\MH$ of isometric copies of $H$ in $(P_{2l})^{k - 1} \times
		\{p, p + 1\}$ for which every point in $(P_{2l})^{k - 1} \times \{p\}$ is
		covered $(1 \mod l)$ times. Let $\MH'$ be the collection of isometric
		copies of $H$ in $(P_{2l})^{k - 1} \times \{p - 1, p\}$, which is the
		image
		of $\MH$ under the map from $(P_{2l})^{k - 1} \times \{p, p + 1\}$ to
		$(P_{2l})^{k - 1} \times \{p -1, p\}$ obtained by changed the last
		coordinate from $p + 1$ to $p - 1$.

		\remark{is there a nicer way to write this?}

		Denote by $\MH_p$ the collection $(l - 1)\MH + \MH'$ (i.e.~each copy of
		$H$ in $\MH$ is taken $l - 1$ times).  $\MH$ is a collection of copies
		of $H$ in $(P_{2l})^{k-1} \times \{p-1, p, p+1\}$ which we view as a
		collection of copies of $H$ in $(P_{2l})^k$. For every $x
		\in (P_{2l})^k$, the number of times $x$ is covered is 
		\begin{align*}
			w_p(x) = \left\{
				\begin{array}{ll}
					(1 \mod l)	&	x \in (P_{2l})^{k - 1} \times \{p + 1\} \\
					(-1 \mod l)	&	x \in (P_{2l})^{k - 1} \times \{p - 1\} \\
					(0 \mod l) 	&	\text{otherwise}
				\end{array}
			\right.
		\end{align*}

		Let $\G$ be the
		collection of isometric copies of $H$ obtained by taking $i$ copies
		of $\MH_{2i}$ and $\MH_{2i - 1}$ for each $i \in [l - 1]$. 
		We show that every vertex in $(P_{2l})^n$ is covered $(1
		\mod l)$ time by $\G$.
		Let $x \in
		(P_{2l})^{k - 1}$. Then $(x, 0)$ and $(x, 1)$ are covered $(1 \mod l)$
		times (they get non zero weight only in $\MH_1$ and $\MH_2$
		respectively). The vertices $(x, 2i)$ and $(x, 2i+1)$ (where $i \in
		[l - 2]$) are covered
		$(-i \mod l)$ times by $\MH_{2i - 1}$ and $H_{2i}$ respectively, and
		$(i + 1 \mod l)$ times by $H_{2i + 1}$ and $H_{2i + 2}$, so in total
		they are covered $(1 \mod l)$ times. Finally, $(x, 2l - 2)$ and $(x,
		2l - 1)$ are covered $l - 1$ times by $H_{2l-3}$ and $H_{2l-2}$
		respectively, so the number of times they are covered is $(-(l-1)
		\mod l) = (1 \mod l)$.
	\end{proof}
		
\section{Almost partitioning the hypercube into powers of a path}
	\label{sec:tiling-with-powers-paths}

	Our main aim in this section is to prove
	\Cref{thm:almost-tiling-induced-paths}.  \thmAlmostTilingInducedPaths*
	Before proceeding to the proof, we make several observations and mention
	a result that we shall use.  The following observation makes use of the
	fact that $Q_n$ is Hamiltonian to prove the special case of
	\Cref{thm:almost-tiling-induced-paths} where $l$ is odd and the
	requirement that the copies are induced is dropped.

	\begin{obs} \label{obs:tiling-odd-paths}
		Let $l$ be odd, and let $t$ be a positive integer. 
		Then there exists a $(P_l)^t$-packing of $Q_n$, that covers all but
		at most $O(n^{t-1})$ vertices.
	\end{obs}

	\begin{proof}
		It is well known that $Q_m$ is Hamiltonian. Therefore, if $l$ divides
		$2^m - 1$, then $Q_m$ may be partitioned into copies of $P_l$ and a
		single vertex. Note that by the Fermat-Euler theorem, there exists
		$m$ such that $l$ divides $2^m - 1$ (take $m = \phi(l)$, where
		$\phi(n)$ is Euler's totient function, counting the number of
		integers $p < n$ for which $p$ and $n$ are coprime).  It follows that
		all but at most $2^{m(t-1)} \cdot |[r]^{< t}|$ vertices of $(Q_m)^r$
		can be partitioned into copies of $(P_l)^t$.  Given any $n$, write $n
		= rm + a$ where $a < r$. Since we may view $Q_n$ as $(Q_m)^{r} \times
		Q_a$, there is a collection of pairwise disjoint induced copies of
		$(P_l)^t$ that covers all but at most $2^{a + m(t-1)} \cdot
		|[r]^{<t}|\le 2^m \cdot n^{t - 1} = O(n^{t - 1})$ vertices.
	\end{proof}

	The following corollary allows us to extend \Cref{obs:tiling-odd-paths}
	to all $l$.

	\begin{cor} \label{cor:tiling-even-paths}
		Let $l$ and $t$ be integers. Then all but $O(n^{t-1})$ vertices
		of $Q_n$ may be partitioned into copies of $(P_l)^t$.
	\end{cor}

	\begin{proof}
		We prove the statement by induction on $i$, the maximum power of $2$
		that divides $l$.  If $i = 0$, $l$ is odd, and the statement follows
		from \Cref{obs:tiling-odd-paths}.  Now suppose that $i \ge 1$.  Write
		$l = 2k$ and $Q_n = Q_{n - t} \times Q_t$.  By induction, there is a
		collection of pairwise disjoint copies of $(P_k)^t$ that covers all
		but at most $O(n^{t - 1})$ vertices in $Q_{n-t}$.  Note that the
		product $P_k \times Q_1$ is Hamiltonian, i.e.~it spans a $P_{2k} =
		P_l$.  We conclude that $Q_n$ may be covered by pairwise disjoint
		copies of $(P_i)^t$ and a remainder of at most
		$O(n^{t-1} \cdot 2^t) = O(n^{t-1})$ vertices.
	\end{proof}

		Let $\MH_l$ be the collection of graphs on $l$ vertices which have a
		Hamilton path, and let $(\MH_l)^t = \{H_1 \times \ldots \times H_t :
		H_i \in \MH\}$.
		We note that the proofs of
		\Cref{obs:tiling-odd-paths,cor:tiling-even-paths} actually give
		the following slightly stronger statement.

	\begin{cor}
		Let $l$ and $t$ be integers. Then there is a collection of
		pairwise disjoint copies of graphs in $(\MH_l)^t$ that covers all
		but at most $O(n^{t-1})$ vertices of $Q_n$.
	\end{cor}

	In order to obtain an almost-partition into induced copies of
	$(P_l)^t$, we need two more ingredients.
	One is the following observation.

	\begin{obs} \label{obs:hamilton-times-induced-path}
		Let $H \in H_l$. Then $H \times P_{l - 1}$ may be partitioned
		into induced copies of $P_l$.
	\end{obs}

	\begin{proof}
		Denote the vertices of $H$ by $[l]$ and suppose
		that $(1,\ldots,l)$ is a path in $H$.
		Similarly, we denote the vertices of $P_{l-1}$ by $[l-1]$. Let
		$Q_i$ be the path $((i, 1), \ldots, (i, l-i),
		(i+1, l-i), \ldots, (i+1, l-1))$, for $i \in [l-1]$ (see
		\Cref{fig:ham-times-induced}).
		It is easy to see that each $Q_i$ is an induced $P_l$.
	\end{proof}

	\begin{figure}
		\centering
		\includegraphics[scale = 1]{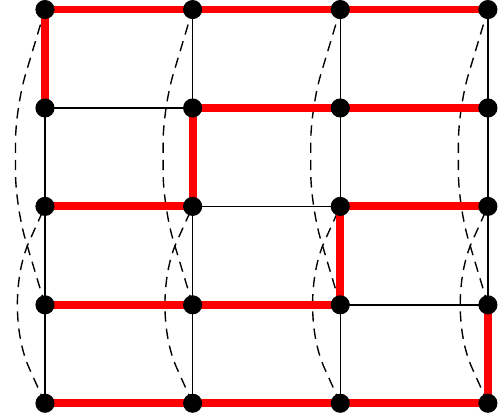}
		\caption{Illustration of the paths in
			\Cref{obs:hamilton-times-induced-path} (vertical lines denote
			copies of $H$)}
		\label{fig:ham-times-induced}
	\end{figure}

	The second ingredient is a result of Ramras \cite{Ramras92}, which
	states that if $n + 1$ is a power of $2$, then $Q_n$ may be
	partitioned into antipodal paths. In particular, we have the
	following corollary.
	
	\begin{cor}[Ramras \cite{Ramras92}] \label{cor:antipodal-paths} 
		If $n + 1$ is a power of $2$, then $Q_n$ may be
		partitioned into induced copies of $P_{n + 1}$.
	\end{cor}

	We are now ready to prove \Cref{thm:almost-tiling-induced-paths}.

	\begin{proof}[ of \Cref{thm:almost-tiling-induced-paths}]
		Let $m$ be minimal such that $2^m \ge l^2$.
		Suppose that $2^m = a (\mod l)$ where $0 \le a < l$ and write $2^m=
		(l-1-a)(l-1) + b$. Note that $(l-1-a)(l-1) < l^2$ so $b > 0$, and
		by choice of $m$, $b \le 2l^2$.
		Furthermore, $b = -1 (\mod l)$.
		The path $P_{2^m}$ can be partitioned into $l-1-a$ copies of
		$P_{l-1}$ and one copy of $P_{b}$.
		It follows from \Cref{cor:antipodal-paths} that $Q_{2^m-1}$ may be
		partitioned into induced copies of $P_{l-1}$ and $P_b$, implying
		that $(Q_{2^m-1})^{2t}$ may be partitioned into induced copies of
		$(P_{l-1})^t$ and $(P_b)^t$.

		Write $Q_n = Q_{n - t (2^m-1)} \times (Q_{2^m-1})^{2t}$.  Recall that
		by \Cref{cor:tiling-even-paths}, $Q_{n - t (2^m-1)}$  may be
		partitioned into copies of graphs in $(\MH_{b+1})^t$ and a remainder
		of at most $O(n^{t-1})$ vertices.

		Combining these two facts, we conclude that $Q_n$ can be partitioned
		into copies of graphs isomorphic to $\prod_{i \in [t]}(H_i \times
		P_x)$ (where $H_i \in \MH_{b+1}$ and $x \in \{l-1, b\}$), and a
		remainder of at most $O\left(n^{t-1} \cdot 2^{t (2^m-1)}\right) =
		O(n^{t-1})$.  We claim that if $H \in \MH_{b+1}$ and $x
		\in \{l-1, b\}$, then $H \times P_x$ may be partitioned into induced
		copies of $P_l$.  Indeed, if $x = b$ then by
		\Cref{obs:hamilton-times-induced-path}, $H \times P_x$ may be
		partitioned into induced $P_{b+1}$'s, which may in turn be
		partitioned into induced $P_l$'s (since, by choice of $b$, $l$ divides
		$b+1$). A similar argument holds if $x = l-1$: first note that $H$
		may be partitioned into graphs in $\MH_l$ (since $l$ divides $b+1$),
		and the product of these graphs with $P_{l-1}$ may be partitioned
		into induced copies of $P_l$, by
		\Cref{obs:hamilton-times-induced-path}.  It follows that $Q_n$ may be
		partitioned into induced copies of $(P_l)^t$ and a remainder of order
		at most $O(n^{t-1})$.
	\end{proof}

\section{A lower bound on the number of uncovered vertices}
	\label{sec:missing-vs}
	In this section we prove \Cref{thm:missing-vs}.
	\thmMissingVs*

	We start by proving the following propositions that characterises the
	intersection of a copy of $(P_3)^t$ in $Q_n$ with a subcube of
	co-dimension $1$.
	
	\begin{prop} \label{prop:intersect-codim-one}
		Let $H$ be a copy of $(P_3)^k$ in $Q_n$. Then the intersection of
		$H$ with any subcube $S$ of co-dimension $1$ is a copy of one of
		the following graphs: $\emptyset$, $(P_3)^{k-1}$, $P_2 \times
		(P_3)^{k-1}$ or $(P_3)^k$.
	\end{prop}

	\begin{proof}
		Let $S$ be the vertex set of a subcube of $Q_n$ of co-dimension $1$.
		Write $H = H' \times P_3$, where every $H'$ is a copy of
		$(P_3)^{k-1}$, and denote $H_i = H' \times \{i \}$ (where $V(P_3) = 
		\{1,2,3\}$. We prove the
		statement by induction on $k$.

		Let $k = 1$, then each $H_i$ is a single vertex.  Without loss of
		generality, $H_2$ is in $S$ (otherwise consider the complement of
		$S$). But then at least one of $H_1$ and $H_3$ also are in $S$
		(because every vertex in $S$ has exactly one neighbour outside of
		$S$). So, without loss of generality, $H_1$ is in $S$. It follows
		that $V(H) \cap S$ is either $H$ or $H_1 \times H_2$, as claimed.

		Now suppose that $k \ge 2$.  Then by induction, and without loss of
		generality, the intersection of $S$ with $H_1$ is either $H_1$ or a
		copy of $P_2 \times (P_3)^{k-1}$.

		Suppose that the first case holds, i.e.~the intersection of $S$ with
		$H_1$ is $H_1$. Then, if any vertex in $H_2$ is in $S$, all vertices
		of $H_1$ are in $S$ (since every vertex in $S$ has exactly one
		neighbour outside of $S$). In other words, $H_2$ is either contained
		in $S$ or it is contained in $\bar{S}$, the complement of $S$.  If
		the former holds, then, similarly, $H_2$ is contained in either $S$
		or $\bar{S}$, and if the latter holds then $H_2$ is contained in
		$\bar{S}$ (since every vertex in $H_2$ is in $\bar{S}$ and has a
		neighbour in $S \cap H_1$).  It follows that the intersection of $S$
		with $H$ in this case is $H_1$, $H_1 \times H_2$, or $H$, as
		required.

		Now suppose that the second case holds, i.e.~the intersection of
		$H_1$ with $S$ is a copy of $P_2 \times (P_3)^{k-1}$. Then we may
		write $H_1 = H'' \times P_3$ where $H''$ is a copy of
		$(P_3)^{k-2}$, and $H'' \times \{1,2\}$ (where $V(P_3) = \{1,2,3\}$)
		is the intersection of $H_1$ with $S$.
		Denote $H_{i,j} = H'' \times \{i\} \times \{j\}$.
		So $H_{1,1}$ and $H_{1,2}$ are in $S$ and $H_{1,3}$ is in $\bar{S}$.
		It follows that $H_{2,2}$ is in $S$
		(otherwise some vertex in $H_{1,2}$ would have two neighbours in
		$\bar{S}$); $H_{2,1}$ is in $S$ (otherwise a vertex of $\bar{S} \cap
		H_{2,1}$ would have two neighbours in $S$); and $H_{2,3}$ is in
		$\bar{S}$.  Similarly, $H_{3,1}$ and $H_{3,2}$ are contained in $S$
		and $H_{3,3}$ is in $\bar{S}$.  It follows that the intersection of
		$H$ with $S$ is a copy of $P_2 \times (P_3)^{k-1}$.
	\end{proof}

	We are now ready for the proof of \Cref{thm:missing-vs}
		
	\begin{proof}[ of \Cref{thm:missing-vs}]
		Let $\MH$ be a collection of pairwise disjoint copies of $(P_3)^3$ in
		$Q_n$ and let $S$ be a subcube of co-dimension $2$ in $Q_n$, and let
		$S'$ be a subcube of co-dimension $1$ in $Q_n$ that contains $S$.  Let
		$H \in \MH$.  Then, by \Cref{prop:intersect-codim-one}, the
		intersection of $H$ with $S'$ is either the
		empty set or it is the disjoint union of up to three copies of
		$(P_3)^2$.  It follows from
		\Cref{prop:intersect-codim-one} that the intersection of $H$ with $S$
		is the disjoint union of copies of $P_3$.  In particular, since $3$
		does not divide the order of $S$, at least one vertex in $S$ is not
		covered by $\MH$.

		Let $\P$ be the collection of subsets of $[n]$ that correspond to
		vertices of $Q_n$ that are not covered by $\MH$ (where we consider the
		usual map between $Q_n$ and $\P([n])$ that sends a vertex $u$ in $Q_n$ to
		the set of elements in $[n]$ whose coordinates in $u$ is $1$).
		We claim that the collection $\P$ is a \emph{separating family} for
		$[n]$, namely, for every distinct elements $i$ and $j$ in $[n]$,
		there is a set $A \in \P$ that contains $i$ but not $j$.
		Indeed, given distinct $i$ and $j$ in $[n]$, let $S$ be the
		subcube of co-dimension $2$ of vertices whose $i$-th coordinate is $1$
		and whose $j$-th coordinate is $0$. Then $S$ contains a vertex which
		is uncovered by $\MH$. This vertex corresponds to a set in $\P$ that
		contains $i$ but not $j$.
		It is a well known fact that a family that separates $[n]$ has size
		at least $\log n$. It follows $\P$ has size at least $\log n$,
		implying that at least $\log n$ vertices of $Q_n$ are not covered by
		$\MH$.
	\end{proof}

	We remark that by considering $(P_3)^{2k+1}$ packings of $Q_n$, the
	number of missing vertices can be shown to be at least $(1 +
	o(1)) k \log n$ (since the subsets corresponding to the missing vertices
	form a separating system for $[n]^{(k)}$). 
	
\section{Concluding remarks} \label{sec:conclusion}

	We showed that if $H$ is an induced subgraph of $Q_k$ then there exists a
	packing of $Q_n$ into induced copies of $H$, which misses at most
	$O(n^c)$, for $c = c(H)$. On the other hand, we showed that the error
	term cannot be replaced by anything smaller than $\log n$ (or, as we
	remarked in \Cref{sec:missing-vs} by $c \log n$ for any $c$).
	It would be very interesting to close the gap between the two bounds.

	We believe that the upper bound, of $O(n^c)$ is closer to the truth,
	i.e., we believe that there exist graphs $H$ for which at least
	$\Omega(n^c)$ vertices remain uncovered in any $H$-packing of $Q_n$.
	More specifically, it seems plausible to believe that every $(P_3)^k$
	packing of $Q_n$ leaves at least $\Omega(n^{k-1})$ vertices uncovered.
	We thus state the following question.

	\begin{qn}
		Is there a $(P_3)^k$ packing of $Q_n$ for which the number of
		uncovered points is at most $o(n^{k-1})$?
	\end{qn}
	
	In this paper we are interested in $H$-packings of $Q_n$, which can be
	viewed as $(P_2)^n$. It would be interesting to consider the more general
	setting of $H$-packings of $G^n$. We mention a conjecture of Gruslys
	\cite{Gruslys16III}.

	\begin{conj} [Gruslys \cite{Gruslys16III}]
		Let $G$ be a finite vertex-transitive graph, and let $H$ be an
		induced subgraph of $G$.
		Suppose further that $|H|$ divides $|G|$. Then for some $n$ there is
		a perfect $H$-packing of $G^n$.
	\end{conj}
	We note that the conjecture does not hold if we drop the
	vertex-transitivity (see Proposition 9 in \cite{Gruslys16III}).

	In \Cref{sec:intro}, we mentioned a recent result of Tomon
	\cite{Tomon16II}
	who proved that if $P$ is a poset with a minimum and a maximum, then the
	Boolean lattice $2^{[n]}$ can be partitioned into copies of $P$ and a
	remainder of at most $c$ elements, where $c = c(P)$.
	It would be interesting to generalise his result to all posets $P$,
	dropping the requirement of the existence of a minimum and a maximum.
	This would resolve a conjecture of Gruslys, Leader and Tomon \cite{Gruslys16II}.

	\begin{conj} [Gruslys \cite{Gruslys16II}]
		Let $P$ be a poset. Then the Boolean lattice $2^{[n]}$ can be
		partitioned into copies of $P$ and a remainder of at most $c = c(P)$
		elements.
	\end{conj}

	Finally, we mention a question about Hamilton paths of $Q_n$.
	Recall that in order to prove that $Q_n$ can be almost partitioned into
	induced copies of $(P_l)^t$, we first proved this statement without
	requiring the copies to be induced. That followed easily from the fact
	that $Q_n$ is Hamiltonian. We then used such a partition to obtain a
	partition of $Q_n$ into induced copies of $(P_l)^t$. A more direct
	approach could be to find a Hamilton path $P$ in $Q_n$ for which every
	$l$ consecutive vertices induced a $P_l$. We were unable to determine if
	such a Hamilton path exists.
	We thus conclude the paper with the following question.

	\begin{qn}
		Let $l$ be integer. Is it true that for sufficiently large $n$, there
		is a Hamilton path $Q_n$ for which every $l$ consecutive vertices
		induce a $P_l$?
	\end{qn}

    \bibliography{almost-tiling}
    \bibliographystyle{amsplain}
\end{document}